\documentclass{amsart}
\usepackage{epic,amssymb,amsthm,amscd,proof,graphicx}
\newtheorem{theorem}{Theorem}[section]
\newtheorem{proposition}[theorem]{Proposition}
\newtheorem{lemma}[theorem]{Lemma}

\theoremstyle{definition}
\newtheorem{definition}[theorem]{Definition}

\theoremstyle{remark}

\numberwithin{equation}{section}

\def\norm#1{|\!| #1 |\!|}
\begin{document}

\title[A Nonconforming FEM for Fourth Order Curl Equations in $\mathbb{R}^3$]{A Nonconforming Finite Element Method for Fourth Order Curl Equations in $\mathbb{R}^3$}

\author{Bin Zheng}
\address{Center for Computational Mathematics and
Applications, Department of Mathematics,
The Pennsylvania State University,
University Park, PA 16802}
\curraddr{Department of Mathematics,
University of Maryland, College Park, MD 20742}
\email{bzheng@math.umd.edu}
%\thanks{The first author was supported in part by }

\author{Qiya Hu}
\address{LSEC, Institute of Computational Mathematics and Scientific/Engineering Computing, Academy of Mathematics and Systems Science, Chinese Academy of Sciences, Beijing 100080, China.}
\email{hqy@lsec.cc.ac.cn}
\thanks{The second author was supported by The Key Project of Natural Science Foundation of China G10531080, National Basic Research Program of China No. 2005CB321702 and Natural Science Foundation of China G10771178.}

\author{Jinchao Xu}
\address{Center for Computational Mathematics and Applications, Department of Mathematics, The Pennsylvania State University, University Park, PA 16802.}
\email{xu@math.psu.edu}
\thanks{The third author was supported by the National Science Foundation under contract DMS-0609727 and DMS-0915153, and Center for Computational Mathematics and Applications, Penn State University.}

%    General info

\date{January 29, 2010}

%\dedicatory{This paper is dedicated to .}


\begin{abstract}
In this paper we present a nonconforming finite element method for solving fourth order curl equations in three dimensions arising from magnetohydrodynamics models. We show that the method has an optimal error estimate for a model problem involving both $(\nabla\times)^2$ and $(\nabla\times)^4$ operators. The element has a very small number of degrees of freedom and it imposes the inter-element continuity along the tangential direction which is appropriate for the approximation of magnetic fields. We also provide explicit formulae of basis functions for this element.
\end{abstract}

\maketitle

\section{Introduction}
The magnetohydrodynamics (MHD) equations describe
macroscopic dynamics of electrically conducting fluid that moves in a magnetic field. MHD model is governed by Navier-Stokes equations coupled with Maxwell equations through Ohm's law and Lorentz force. As an example, a resistive MHD system is described by the following equations:
\begin{equation*}
\left\{
\begin{array}{rcl}
\rho(\mathbf{u}_t+\mathbf{u}\cdot \nabla \mathbf{u})+\nabla p &=& \frac{1}{\mu_0}(\nabla\times \mathbf{B})\times \mathbf{B} + \mu\Delta \mathbf{u},\\
\nabla\cdot \mathbf{u}  & = & 0,\\
\mathbf{B}_t  - \nabla\times(\mathbf{u}\times \mathbf{B}) & = & -\frac{\eta}{\mu_0}(\nabla\times)^2 \mathbf{B} - \frac{d_i}{\mu_0}\nabla\times((\nabla\times \mathbf{B})\times\mathbf{B}) \\
&& - \frac{\eta_2}{\mu_0}(\nabla\times)^4 \mathbf{B}, \\
\nabla\cdot \mathbf{B}  & = & 0,
\end{array}
\right.
\end{equation*}
where $\rho$ is the mass density, $\mathbf{u}$ is the velocity, $p$ is the pressure, $\mathbf{B}$ is the magnetic induction field, $\eta$ is the resistivity, $\eta_2$ is the hyper-resistivity, $\mu_0$ is the magnetic permeability of free space, and $\mu$ is the viscosity. The primary variables in MHD equations are fluid velocity $\mathbf{u}$ and magnetic field $\mathbf{B}$.

MHD model has widespread applications in thermonuclear fusion, magnetospheric and solar physics, plasma physics, geophysics, and astrophysics. Mathematical modeling and numerical simulations of MHD have attracted much research effort in the past few decades. Various numerical algorithms have been used in MHD simulations; examples include finite difference methods, finite volume methods, finite element methods, and Fourier-based spectral and pseudo-spectral methods \cite{Toth:1996gd}. In \cite{Jardin:2004qo,Jardin:2005dq, Kang:2008ta,Krzeminski:2000la,Ovtchinnikov:2007kb}, two-dimensional, incompressible MHD problems are studied in terms of finite element approximations of the stream function-vorticity advection formulation. Since MHD flow often develop sharp interfaces, adaptive $h$-refinement techniques have been applied in MHD simulations \cite{Lankalapallia:2007ye, Strauss:1998kc, Ziegler:2003gf}. Finite element computations of MHD problems in three-dimensions have been reported in \cite{Codina:2006zv, Gerbeau:2000lo, Layton:1997hs,Salah:2001fh,Schotzau:2004sw, Wiedmer:1999pr}.

In the existing finite element discreitzations for the above MHD model, a standard pair of stable or stabilized finite element spaces are often used to discretize the velocity and pressure variables in the fluid equations. For the magnetic field variable $\mathbf{B}$, however, at least two approaches are possible when the fourth order term $(\nabla\times)^4\mathbf{B}$ is not presented in the model, namely when electron viscosity $\eta_2=0$. One approach is to use the standard edge element (\cite{Schotzau:2004sw}) and the other approach is to use the Lagrange element after replacing $(\nabla\times)^2\mathbf{B}$ by $-\Delta \mathbf{B}$ (\cite{Gerbeau:2000lo, Salah:2001fh,Wiedmer:1999pr}). Both these approaches will become more difficult when the fourth order term $(\nabla\times)^4\mathbf{B}$ is presented. We may still replace $(\nabla\times)^4 \mathbf{B}$ by a biharmonic operator $\Delta^2 \mathbf{B}$. But discretizing a biharmonic operator in three dimensions is challenging. It requires $220$ degrees of freedom per element if a conforming finite element is used. One possible way to reduce the number of degrees of freedom is to use nonconforming discretizations which allow weaker inter-element smoothness constraints but still provid convergent approximations. Among the class of nonconforming finite elements for fourth order problems, Morley-type elements are special in the sense that they provide approximations with polynomials of minimal degree \cite{ Morley:1968ve, Wang:2006qr}. In \cite{Wang:2006bh}, a systematic construction of Morley-type elements is provided for solving $2m$-th order partial differential equations in $\mathbb{R}^n$.
In particular, we may apply the element in \cite{Wang:2006bh} with $n=3$ and $m=2$ consisting of piecewise quadratic elements to our system of biharmonic equations. This amounts to $30$ degrees of freedom on each element.  This element provides a reasonable discretization of MHD equations when it
is appropriate to replace $(\text{curl})^4$ operator by the biharmonic operator.  This approach, however, may lead to difficulty for certain boundary conditions in practical applications. Indeed the treatment of boundary conditions is also an issue for the second order problem if $(\nabla\times)^2\mathbf{B}$ is replaced by $-\Delta\mathbf{B}$ \cite{Guermond:2003fk}.

One more natural approach is to discretize the fourth order curl operator by some generalized higher order edge elements. But such type of edge elements are not available in the literature. The construction of such type of edge element is the main goal of this paper.

Another possible approach to deal with the fourth order term $(\nabla\times)^4\mathbf{B}$ is to use operator splitting technique. Namely, one can introduce an intermediate variable $\sigma = (\nabla\times)^2\mathbf{B}$ and then reduce the original problem to a system of second order equations. However, it is known that for some problems, such a technique cannot be applied. For example, when modeling the bending of simply supported plate on non-convex polygonal domains, the original biharmonic problem is not equivalent to the lower order system of two Poisson
equations \cite{Blum:1980ly, Rannacher:1979zr}. In view of this, we consider discretizing the fourth order problem directly.

In this paper, we investigate MHD equations that contain both fourth-order term and second-order term. In the literature, the major tool used for performing MHD simulations involving a fourth order equation has been the pseudo-spectral method \cite{Biskamp:1995it}. By choosing an appropriate formulation, we are able to construct a finite element approximation for this problem. This is a nonconforming finite element
that involves only a small number of degrees of freedom.

The rest of this paper is organized as follows. In Section 2, we
describe a simplified model problem and the corresponding variational formulation. In Section 3, we construct basis functions and provide the convergence analysis. Finally, in Section 4, we give some concluding remarks.

\section{Model Problem}
In the following, we introduce model problems for the fourth-order magnetic induction equations described above. Assume that $\Omega\subset \mathbb{R}^3$ is a bounded polyhedron. By considering a semi-discretization in time and then ignoring the nonlinear terms, we obtain the following equations:
\begin{equation}
\left\{
\begin{aligned}
\alpha(\nabla\times)^4 \mathbf{u} + \beta(\nabla\times)^2 \mathbf{u} +\gamma \mathbf{u}&=\mathbf{f},\;\text{in}\;\Omega,\\
\nabla\cdot\mathbf{u} & = 0,\;\text{in}\;\Omega,
\end{aligned}
\right.\label{four_curl_bvp}
\end{equation}
where $\nabla\cdot\mathbf{f}= 0$, and the parameters $\alpha, \beta, \gamma > 0 $. We consider homogeneous boundary conditions,
\begin{equation}
\mathbf{u}\times \mathbf{n}= 0,\;\nabla\times \mathbf{u}= 0,\;\text{on}\;\partial\Omega.
\end{equation}

The above choice of boundary conditions arise naturally in the variational formulation given below. On the other hand, in the numerical simulations of the problem with pseudo-spectral method, one often uses periodic boundary conditions, e.g., \cite{Biskamp:1995it,Germaschewski:1999yq}.

It is worth pointing out that the parameter $\alpha$ is usually much smaller than either $\beta$ or $\gamma$. This fact imposes some difficulties in designing robust numerical methods, as have been studied in the context of biharmonic problems, e.g., \cite{Nilssen:2001zr,Wang:2006ly}.

The above fourth-order curl equations also arise from an interior transmission problem in the study of inverse scattering problems for inhomogeneous medium, e.g., \cite{Cakoni:2007yq}.

In order to provide an appropriate framework for our analysis, we define the following function spaces:
$$
H(\text{curl};\Omega)=\{\mathbf{v}\in (L^2(\Omega))^3 \; | \;\nabla\times \mathbf{v}\in (L^2(\Omega))^3\},
$$
$$
H_0(\text{curl};\Omega) = \{\mathbf{v}\in H(\text{curl};\Omega)\; | \;\mathbf{v}\times \mathbf{n} = 0, \text{on}\;\partial\Omega\},
$$
$$
V=\{\mathbf{v}\in H_0(\text{curl};\Omega)\; |\; \nabla\times \mathbf{v}\in H_0^1(\Omega)\}.
$$
$V$ is a Hilbert space with scalar product and norm given by
$$
(\mathbf{u},\mathbf{v})_{V}\triangleq (\nabla(\nabla\times \mathbf{u}),\nabla (\nabla\times
\mathbf{v}))+(\nabla\times \mathbf{u},\nabla\times \mathbf{v})+(\mathbf{u},\mathbf{v}),
$$
$$
\norm{\mathbf{u}}_{V}\triangleq \sqrt{(\mathbf{u},\mathbf{u})_{V}}.
$$

The following lemma gives a sufficient condition for a piecewisely defined function to be an element in $V$.

\begin{lemma}
If $\mathbf{v}$ is piecewise smooth, $\mathbf{v}\times \mathbf{n}$ and $\nabla\times \mathbf{v}$ are continuous across element interfaces, then
$\mathbf{v}\in V$.
\end{lemma}

Using the following identity:
\begin{equation*}
(\nabla\times)^2\mathbf{u}=-\Delta\mathbf{u}+
\nabla(\nabla\cdot\mathbf{u})
\end{equation*}
and $\nabla\cdot \mathbf{u}=0$, the first equation in (\ref{four_curl_bvp}) can be rewritten in the following form:
\begin{equation}
-\alpha\nabla\times\Delta(\nabla\times \mathbf{u}) + \beta(\nabla\times)^2 \mathbf{u} +\gamma \mathbf{u}=\mathbf{f}.
\label{delta_curlcurl}
\end{equation}

Multiplying Equation (\ref{delta_curlcurl}) by the test function $\mathbf{v}$ and using integration by parts, we obtain the following variational formulation:
\begin{equation}
\text{Find}\; \mathbf{u}\in V\; \text{such that}\; a(\mathbf{u},\mathbf{v})=(\mathbf{f},\mathbf{v}), \;\forall\;\mathbf{v}\in V,\label{V_1_variational}
\end{equation}
where the bilinear form $a(\cdot,\cdot)$ defined on $V\times
V$ is given by
$$
a(\mathbf{u},\mathbf{v})=\alpha (\nabla(\nabla\times \mathbf{u}),\nabla(\nabla\times \mathbf{v}))+ \beta (\nabla\times
\mathbf{u},\nabla\times \mathbf{v}) + \gamma (\mathbf{u},\mathbf{v}).
$$

The well-posedness of the above variational problem follows from the
Lax-Milgram lemma.

The next lemma indicates that the weak solution satisfies the divergence-free constraint.

\begin{lemma}
Assume $\nabla\cdot \mathbf{f} = 0$, and let $\mathbf{u}$ be the solution of problem (\ref{V_1_variational}). Then $\nabla\cdot \mathbf{u} = 0$.
\end{lemma}

\begin{proof}
Choose test function $\mathbf{v}=\nabla \varphi$ where $\varphi\in C_0^\infty(\Omega)$, then
$$
(\mathbf{u},\nabla\varphi) = (\mathbf{f},\nabla\varphi),
$$
hence, $\nabla\cdot \mathbf{u} =\nabla\cdot \mathbf{f} = 0$.
\end{proof}

\section{A Nonconforming Finite Element}

In this section, we construct a nonconforming finite element to solve the fourth-order equation. One of the advantages for using a nonconforming element is that the number of degrees of freedom is small compared to that for conforming elements. The following construction is based on N\'{e}d\'{e}lec elements of the first family that consist of incomplete vector polynomials \cite{Nedelec:1980xe}. The advantage of using incomplete vector polynomial space is that it provides the same order of convergence in terms of energy norms as the one given by corresponding complete polynomial space. In the following, we define the degrees of freedom in a special way to ensure that the consistency error estimate holds.

\begin{definition}
The finite element triple
$(K,P_K,\Sigma_K)$ is defined by

\begin{itemize}
\item $K$ is a tetrahedron;

\item $\mathcal{P}_K=R_2(K)=\mathbf{P}_{1}\oplus \{\mathbf{p}\in
(\widetilde{P}_2)^3\;|\;\mathbf{p}\cdot \mathbf{x} = 0\}$, where $\widetilde{P}_2$ is the space of homogeneous multivariate polynomials of degree $2$;

\item $\Sigma_K$ is the set of degrees of freedom, see Figure \ref{element},
\begin{itemize}
\item edge degrees of freedom:
\begin{equation}
M_e(\mathbf{u})=\bigg\{\int_e \mathbf{u}\cdot \tau\;q\;d s\;|\;\forall
\;q\in P_1(e),\;\forall\;e\subset K\bigg\},
\label{edge d.o.f}
\end{equation}
where $\tau$ is the unit tangential vector along the edge $e$,

\item face degrees of freedom:
\begin{equation}
M_f(\mathbf{u})=\bigg\{\frac{1}{|f|^2}\int_f (\nabla\times \mathbf{u})\times \mathbf{n} \cdot
q \;d A\;|\;\forall\;q\in
(P_0(f))^2,\;\forall\;f\subset K\bigg\},
\label{face d.o.f}
\end{equation}
where $\mathbf{n}$ is the unit normal vector to the face $f$,

\end{itemize}
$\Sigma_K=M_e(\mathbf{u})\cup M_f(\mathbf{u})$.
\end{itemize}
\label{fem_triple_def}
\end{definition}

\begin{figure}[h]
\begin{center}
   \includegraphics[width=60mm]{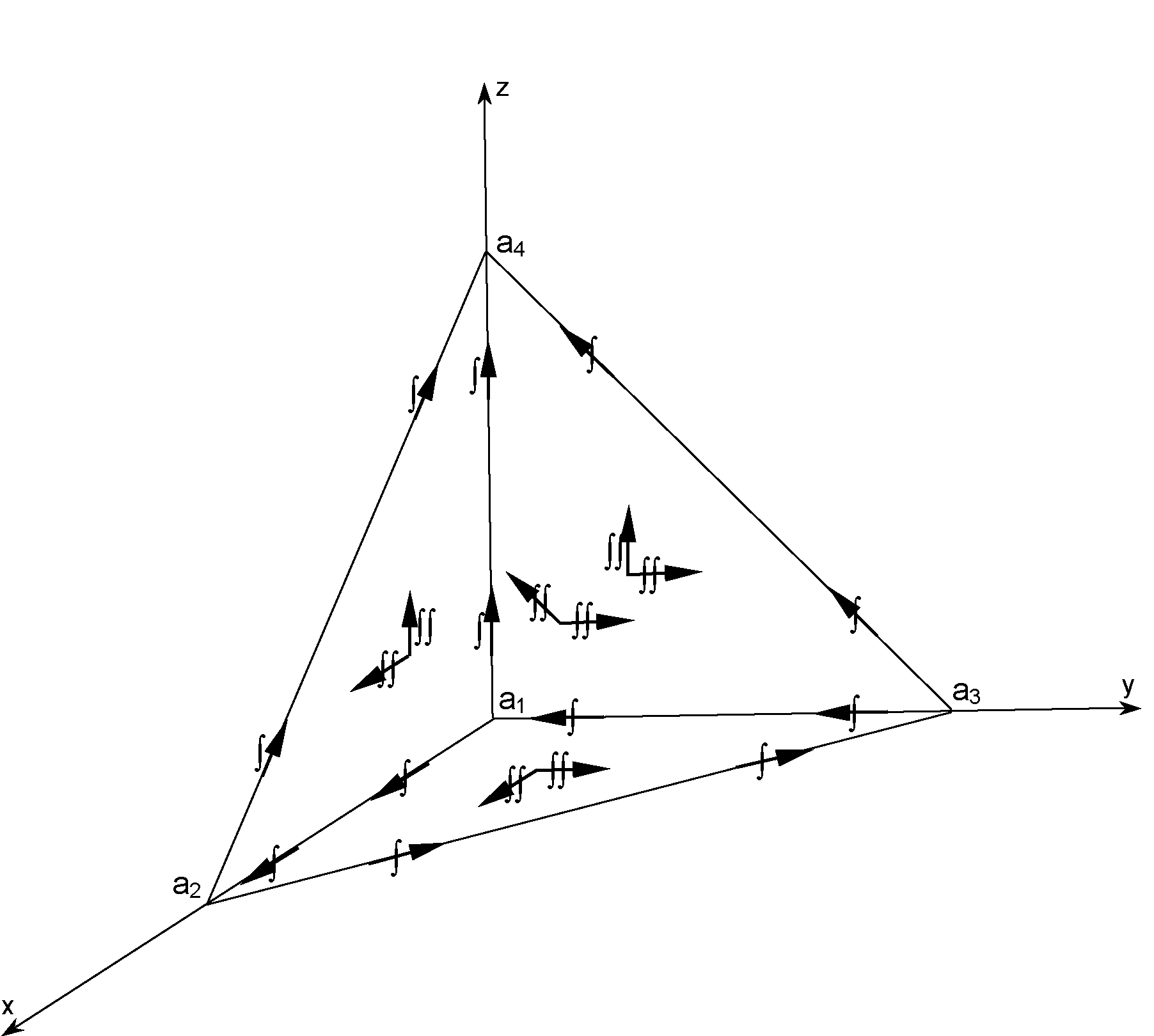}
   \caption{Degrees of freedom of the finite element}
   \label{element}
\end{center}
\end{figure}

In the above finite element triple, the space $\mathcal{P}_K$ is the same as the second order N\'{e}d\'{e}lec element of the first family for $H(\text{curl})$ problem. The difference is the definition of the second set of degrees of freedom. It is designed specifically to ensure consistency for the fourth-order problems. The total number of the degrees of freedom for this element is $20$, which is the same as the dimension of the polynomial space $R_2(K)$.

It should be pointed out that the scaling factor $1/|f|^2$ in the definition of the second set of degrees of freedom is associated with the construction of basis functions to be given later.

The next lemma given in \cite{Nedelec:1980xe} describes a relation between edge integrals and face integrals which will be useful in the error analysis.

\begin{lemma}
If $u\in R_2(K)$ is such that the edge degrees of freedom (\ref{edge d.o.f}) vanish, then
\begin{equation}
\int_f(\nabla\times \mathbf{u})\cdot \mathbf{n}\;dA=0,\;\forall\;\text{face}\;f\subset K.
\label{curl_u_face_integral}
\end{equation}
\label{relation-edge-face}
\end{lemma}
\begin{proof}
Given $\mathbf{u}\in R_2(K)$ satisfies
$$
\int_e \mathbf{u}\cdot\tau \;q\;ds=0,\;\forall \;\text{edge}\; e\subset K.
$$
By Stokes' Theorem,
$$
\int_f (\nabla_f\times \mathbf{u}_T)\cdot q \;dA- \int_f (\vec{\nabla}_f\times q)\cdot \mathbf{u}_T\;dA = \int_{\partial f} \mathbf{u}\cdot \tau\; q\;ds,
$$
where $\mathbf{u}_T$ is the tangential part of $\mathbf{u}$, and $\vec{\nabla}_f\times$ and $\nabla_f\times$ are surface vector curl and surface scalar curl, respectively. Let $q$ be a constant. Notice that
$$
\nabla_f \times \mathbf{u}_T = (\nabla\times \mathbf{u})\cdot \mathbf{n},
$$
we conclude,
$$
\int_f(\nabla\times \mathbf{u})\cdot \mathbf{n}\;dA=0.
$$
\end{proof}

As a direct consequence of Lemma \ref{relation-edge-face}, if both the edge degrees of freedom (\ref{edge d.o.f}) and face degrees of freedom (\ref{face d.o.f}) vanish, then
$$
\int_f (\nabla\times \mathbf{u}) \;dA= 0.
$$

The polynomial space $R_2(K)$ has the following property \cite{Girault:1986fk}.

\begin{lemma}
If $\mathbf{u}\in R_2(K)$ satisfies $\nabla\times \mathbf{u} = 0$, then
$$
\mathbf{u}=\nabla p,\;\text{with}\;p\in P_2.
$$
\label{R2lemma}
\end{lemma}

We recall that the finite element $(K,\mathcal{P}_K,\Sigma_K)$ is said to be unisolvent if a function in $\mathcal{P}_K$ can be uniquely determined by specifying values for degrees of freedom in $\Sigma_K$.

\begin{proposition}
The finite element defined by Definition \ref{fem_triple_def} is unisolvent.
\end{proposition}

\begin{proof}
It is sufficient to prove that, given $\mathbf{u}\in R_2(K)$,
$$
M_e(\mathbf{u}) = M_f(\mathbf{u}) = 0,\;\forall e\subset K, f\subset K\Rightarrow \mathbf{u} = 0.
$$
Obviously, $
\nabla (\nabla\times \mathbf{u})$ is a constant vector.
Then using (\ref{curl_u_face_integral}) and integration by parts, we obtain
$$
\nabla(\nabla\times \mathbf{u}) = \frac{1}{|K|}\int_K \nabla(\nabla\times \mathbf{u}) dx  = \frac{1}{|K|}\int_{\partial K} (\nabla\times \mathbf{u})\mathbf{n}^T dA = 0.
$$
This implies that
$$
\nabla(\nabla\times \mathbf{u}) = 0 \Rightarrow \nabla\times \mathbf{u} = \text{const}.
$$

Using again (\ref{curl_u_face_integral}), we have
\begin{equation*}
\nabla\times \mathbf{u} = 0.
\end{equation*}
By Lemma \ref{R2lemma}, we have
$$
\mathbf{u}=\nabla p,\;\text{with}\;p\in P_2(K).
$$

Since $M_e(\mathbf{u}) = 0$, we have
$$
\int_e \frac{\partial p}{\partial \tau}\; q\; ds = 0, \forall q\in P_1(e).
$$
This implies $\partial p/\partial \tau = 0$ on each edge $e$. Hence, $p$ is constant and $\mathbf{u}=0$.
\end{proof}

In the following, we construct the basis functions. The explicit form of these basis functions not only is useful for implementation, but also instrumental for the interpolation error estimate.

\subsection{Basis functions}
 The main idea of the construction is to consider linear combinations of basis functions of a related N\'{e}d\'{e}lec element. Let $K$ be an arbitrary tetrahedron with four vertices $a_i$, $a_j$, $a_k$ and $a_l$, see Figure \ref{element2}. The corresponding barycentric coordinates are given by
$\lambda_i$, $\lambda_j$, $\lambda_k$, and $\lambda_l$, respectively.

\begin{figure}[h]
\begin{center}
   \includegraphics[width=60mm]{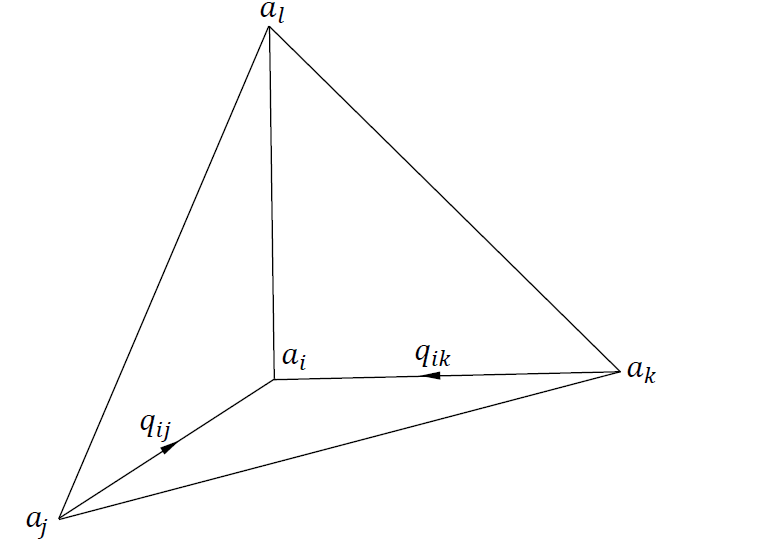}
   \caption{A tetrahedron with vertices $a_i$, $a_j$, $a_k$, $a_l$. $\mathbf{q}_{ij}$ and $\mathbf{q}_{ik}$} are two tangential vectors on the face $F_l$.
   \label{element2}
\end{center}
\end{figure}

On each of the four faces, say face $l$ (with vertices $a_i, a_j, a_k$), we choose the following two tangential direction vectors:
$$
\mathbf{q}_{ij} = \overrightarrow{a_j a_i} = 6|K|(\nabla\lambda_l\times \nabla\lambda_k),
$$
$$
\mathbf{q}_{ik} = \overrightarrow{a_k a_i} =
6|K|(\nabla\lambda_j\times \nabla\lambda_l).
$$

The edge degrees of freedom on edge $e_{ij}$ (with vertices $a_i$ and $a_j$) are defined explicitly by:
$$
M_{ij}^{(1)}(\mathbf{u}) = \int_{e_{ij}} \mathbf{u}\cdot\tau \;ds,
$$
$$M_{ij}^{(2)}(\mathbf{u}) = \int_{e_{ij}} \mathbf{u}\cdot
\tau\left(3-\frac{6}{|e_{ij}|}s\right)\;ds,
$$
where $\tau$ is the unit direction vector of edge $e_{ij}$, $s$ is an arc length parameter. The face degrees of freedom are defined as:
$$
M_{lij}(\mathbf{u}) = \frac{1}{|f_l|^2} \int_{f_{l}}(\nabla \times \mathbf{u})\times \mathbf{n}_l\cdot \mathbf{q}_{ij}\;dA,
$$
$$
M_{lik}(\mathbf{u}) = \frac{1}{|f_l|^2} \int_{f_{l}}(\nabla \times \mathbf{u})\times \mathbf{n}_l\cdot \mathbf{q}_{ik}\;dA,
$$
where $\mathbf{n}_l$ is the unit outward normal vector of the face $f_l$.

We recall that the basis functions of the second order N\'{e}d\'{e}lec element of the first family in barycentric coordinates are (see, e.g., \cite{Gopalakrishnan:2005eu}, \cite{Sun:2008yq}, \cite{Webb:1999sh}):

(1) Two basis functions on each edge $e_{ij}$:
$$
\mathbf{L}_{ij} = \lambda_i\nabla\lambda_j-\lambda_j\nabla\lambda_i,
$$
$$
\mathbf{L}_{ji} = \lambda_i\nabla\lambda_j+\lambda_j\nabla\lambda_i.
$$

(2) Two basis functions on each face $f_l$:
$$
\mathbf{L}_{ijk} = \lambda_i(\lambda_j\nabla\lambda_k - \lambda_k\nabla\lambda_j),
$$
$$
\mathbf{L}_{jik} = \lambda_j(\lambda_i\nabla\lambda_k - \lambda_k\nabla\lambda_i).
$$

In the following, we list a few useful facts about the geometry of a tetrahedron.

(1) The unit outward normal vector of face $f_l$ is given by
$$
-\frac{\nabla \lambda_l}{\|\nabla\lambda_l\|}.
$$

(2) The two tangential vectors of face $f_l$ are given by $\mathbf{q}_{ij}$ and $\mathbf{q}_{ik}$.

(3)  Let $h_l$ be the height of the tetrahedron corresponding to the face $f_l$, then
 $$
\nabla\lambda_l=\frac{1}{6|K|}\mathbf{q}_{ik}\times \mathbf{q}_{jk},
$$
$$
|\nabla\lambda_l| = \frac{1}{h_l}.
$$

(4)  Let $|K|$ be the volume of the tetrahedron $K$, then
 $$
6|K|=|\mathbf{q}_{il}\cdot(\mathbf{q}_{jl}\times \mathbf{q}_{kl})| = \frac{-1}{(\nabla\lambda_i\times \nabla\lambda_j)\cdot \nabla\lambda_k }.
$$

Next, we construct basis functions in barycentric coordinates. They provide a set of dual basis functions with respect to the prescribed degrees of freedom.

{\bf Step 1}. Construct eight basis functions $\{\phi_{lij}\}$ corresponding to the face degrees of freedom such that
\begin{equation}
M_{mn}^{(t)}(\phi_{lij})=0,\label{face_function_edge}
\end{equation}
and
\begin{equation}
M_{mnp}(\phi_{lij})=\delta_{ml}\delta_{ni}
\delta_{pj}.\label{face_function_face}
\end{equation}
We use the basis functions of the second order N\'{e}d\'{e}lec element as building blocks as they automatically satisfy the first condition (\ref{face_function_edge}).  Using the facts listed above, we find that the basis functions corresponding to the facial degrees of freedom on face $f_l$ are given by the following:
$$
\phi_{lij} = 3|K|(\mathbf{L}_{lij}-\mathbf{L}_{ljk}),
$$
$$
\phi_{lik} = 3|K|(\mathbf{L}_{lik}-\mathbf{L}_{ljk}).
$$

By direct calculation, we have
\begin{eqnarray*}
&&\int_{f_l}(\nabla\times \mathbf{L}_{lij})\times \nabla\lambda_l\cdot(\nabla\lambda_l\times\nabla\lambda_k)\;dA\\
& = & \int_{f_l}\left[2\lambda_l(\nabla\lambda_i\times\nabla\lambda_j)
+\lambda_i(\nabla\lambda_l\times\nabla\lambda_j)-\lambda_j
(\nabla\lambda_l\times\nabla\lambda_i)\right]\\
&& \cdot\left[\nabla\lambda_l(\nabla_l\cdot\nabla_l)
-\nabla\lambda_k(\nabla\lambda_l\cdot\nabla\lambda_k)
\right]\;dA\\
& = & -\left(\int_{f_l}\lambda_i \;dA\right)[(\nabla\lambda_l\times\nabla\lambda_j)
\cdot \nabla\lambda_k](\nabla\lambda_l\cdot\nabla\lambda_l)\\
&& + \left(\int_{f_l}\lambda_j \;dA\right)[(\nabla\lambda_l\times\nabla\lambda_i)
\cdot \nabla\lambda_k](\nabla\lambda_l\cdot\nabla\lambda_l)\\
& = & -\frac{2}{3}|f_l|\frac{1}{6|K|h_l^2},
\end{eqnarray*}
and
\begin{eqnarray*}
\int_{f_l}(\nabla\times \mathbf{L}_{ljk})\times \nabla\lambda_l\cdot(\nabla\lambda_l\times\nabla\lambda_k)\;dA = \frac{1}{3}|f_l|\frac{1}{6|K|h_l^2}.
\end{eqnarray*}
Hence,
\begin{eqnarray*}
M_{lij}(\phi_{lij})&=&\frac{1}{|f_l|^2}\int_{f_l}
(\nabla\times \phi_{lij})\times n_l\cdot \mathbf{q}_{ij}\;dA\\
&=& \frac{1}{|f_l|^2}\int_{f_l}\nabla\times\left[
3|K|(\mathbf{L}_{lij}-\mathbf{L}_{ljk})\right]
\times -\frac{\nabla\lambda_l}{\|\nabla\lambda_l\|}
\cdot(6|K|\nabla\lambda_l\times
\nabla\lambda_k)dA
\\
&=& -\frac{18|K|^2h_l}{|f_l|^2}\int_{f_l}\nabla\times
(\mathbf{L}_{lij}
-\mathbf{L}_{ljk})\times\nabla\lambda_l\cdot (\nabla\lambda_l
\times\nabla\lambda_k)\;dA\\
&=& -\frac{18|K|^2h_l}{|f_l|^2}
\left(-\frac{2}{3}|f_l|\frac{1}{6|K|h_l^2}
-\frac{1}{3}|f_l|\frac{1}{6|K|h_l^2}\right) = 1.
\end{eqnarray*}

Similarly, $M_{lik}(\phi_{lij}) = 0$, and $M_{f_{l'}}(\phi_{lij})=0$ where $l'\neq l$.

{\bf Step 2}. Construct twelve basis functions $\{\psi_{ij}^{(t)}|1\leq i<j\leq 4, t=1,2\}$ corresponding to the edge degrees of freedom such that
\begin{equation}
M_{mn}^{(t')}(\psi_{ij}^{(t)}) = \delta_{t't}\delta_{mn,ij},\label{edge_function_edge}
\end{equation}
\begin{equation}
M_{mnp}(\psi_{ij}) = 0.\label{edge_function_face}
\end{equation}

Here, we use the edge basis functions of the second order N\'{e}d\'{e}lec element as building blocks since they satisfy condition (\ref{edge_function_edge}). Since $\nabla\times \mathbf{L}_{ji}=0$, $\mathbf{L}_{ji}$ automatically satisfy condition (\ref{edge_function_face}).

For functions $\mathbf{L}_{ij}$, we need to subtract from them a linear combination of face basis functions so that (\ref{edge_function_edge}) and (\ref{edge_function_face}) hold. This can be done because by construction, our face basis functions have no edge moments. This strategy for constructing basis functions can be found in \cite{Gopalakrishnan:2005eu,Sun:2008yq}.

Finally, we can write the basis functions of the new element as the following:

(1) Two basis functions on each face $l$ ($1\leq l\leq 4$):
$$
\phi_{lij} = 3|K|(\mathbf{L}_{lij}-\mathbf{L}_{ljk}),
$$
$$
\phi_{lik} = 3|K|(\mathbf{L}_{lik}-\mathbf{L}_{ljk}),
$$
where $\mathbf{L}_{lij}= \lambda_i(\lambda_j\nabla\lambda_k - \lambda_k\nabla\lambda_j)$.

(2) Two basis functions on each edge $e_{ij}$ ($1\leq i<j\leq 4$):
$$
\psi_{ij}^{(1)}=\mathbf{L}_{ji},
$$
$$
\psi_{ij}^{(2)}=\mathbf{L}_{ij}-\sum M_{mnp} (\mathbf{L}_{ij})\phi_{mnp}.
$$

\subsection{Convergence analysis}

Let $\mathcal{T}_h=\{K_i\}_{i=1}^{N_h}$ be a triangulation
of the domain $\Omega$. On this triangulation we introduce the finite element space $V_h$ and define the discrete norm $\|\cdot\|_h$ by
$$
\| \mathbf{v}\|_h=\left[\sum_{K\in\mathcal{T}_h}
\left(\|\mathbf{v}\|_{0,K}^2+\|\nabla\times \mathbf{v}\|_{0,K}^2+\|\nabla(\nabla\times \mathbf{v})\|_{0,K}^2\right)\right]^{1/2}.
$$

Consider the following discrete bilinear form:
\begin{equation*}
\begin{split}
a_h(\mathbf{u}_h,\mathbf{v}_h)
=\sum_{K\in\mathcal{T}_h} \alpha(\nabla(\nabla\times \mathbf{u}_h),\nabla(\nabla\times
\mathbf{v}_h))_{L^2(K)}&+\beta (\nabla\times \mathbf{u}_h,\nabla\times
\mathbf{v}_h)_{L^2(K)}\\
 &+ \gamma (\mathbf{u}_h,\mathbf{v}_h)_{L^2(K)}.
\end{split}
\end{equation*}
It is straightforward to verify that the bilinear
form $a_h$ satisfies
$$
a_h(\mathbf{v},\mathbf{v})\gtrsim  \|\mathbf{v}\|_h^2,\;\forall\; \mathbf{v}\in V_h,
$$
$$
|a_h(\mathbf{u},\mathbf{v})|\lesssim
\|\mathbf{u}\|_h\|\mathbf{v}\|_h,\;
\forall\;\mathbf{u}\in V+V_h, \mathbf{v}\in V_h.
$$

The nonconforming finite element discretization of problem
(3.10) is:

Find $\mathbf{u}_h\in V_h$, such that for all $\mathbf{v}_h\in V_h$,
\begin{equation}
a_h(\mathbf{u}_h,\mathbf{v}_h)=(\mathbf{f},\mathbf{v}_h).\label{discrete var prob1}
\end{equation}
The convergence of the above finite element approximation can be analyzed through the following second Strang lemma \cite{Ciarlet:1978qf}.
\begin{lemma}
$$
\|\mathbf{u}-\mathbf{u}_h\|_h\lesssim \inf_{\mathbf{v}_h\in V_h}\|\mathbf{u}-\mathbf{v}_h\|_h +
\sup_{\mathbf{w}_h\in V_h}\frac{|a_h(\mathbf{u},\mathbf{w}_h)-
(\mathbf{f},\mathbf{w}_h)|}{\|\mathbf{w}_h\|_h},
$$
where the first term on the right-hand side is called the interpolation error and the second term is called the consistency error.
\end{lemma}

In order to estimate the consistency error we first define an average operator $P_{f}$ on a face $f$ by
$$
P_{f} \mathbf{w} = \frac{1}{|f|}\int_{f} \mathbf{w} \;dA.
$$
Since for any $\mathbf{v}_h\in V_h$, the quantity $\int_f \nabla\times \mathbf{v}_h\;dA$ is continuous, we know that $P_f$ is well-defined for $\nabla\times \mathbf{v}_h$.
The following two lemmas are standard results.
\begin{lemma} Given any face $f\subset K$ and $\mathbf{w}\in (H^1(K))^3$,
$$
\int_{f} |\mathbf{w}-P_{f}\mathbf{w}|^2 \;dA \lesssim h_K|\mathbf{w}|^2_{1,K}.
$$\label{face_estimate_lemma}
\end{lemma}
\begin{lemma}
$$
\int_{\partial K}|\mathbf{w}|^2 \;dA \lesssim
h_K^{-1}\norm{\mathbf{w}}_{0,K}^2+h_K|
\mathbf{w}|_{1,K}^2.
$$\label{trace_lemma}
\end{lemma}

Next, we estimate the interpolation error and consistency error separately.

\subsubsection{Interpolation error estimate}

Let $K$ and $K'_f$ be the two tetrahedra sharing a common face $f$, $r_K$ be the local interpolation operator for the second order N\'{e}d\'{e}lec element of the first family, namely, given $\mathbf{u}\in V$, define $r_K \mathbf{u}$ such that
$$
\int_e r_K \mathbf{u}\cdot \tau\;ds = \int_e \mathbf{u}\cdot\tau\;ds,\;\forall\;\text{edge}\;
e\subset K,
$$
and
$$
\int_f (r_K \mathbf{u}\times \mathbf{n})
\cdot q\;dA = \int_f (\mathbf{u}\times \mathbf{n})\cdot \mathbf{q}\;dA,\;\forall\;\mathbf{q}\in (P_0(f))^2,\;\forall\;\text{face}\;f\subset K.
$$
Define $\mathbf{u}_I\in V_h$ such that
$$
M_e(\mathbf{u}_I)=M_e(r_K \mathbf{u})=M_e(\mathbf{u}),
$$
$$
M_f(\mathbf{u}_I)=[M_f(r_{K} \mathbf{u})+M_f(r_{K'_f} \mathbf{u})]/2.
$$

If $f\subset \partial \Omega$, we set $M_f(\mathbf{u}_I)=M_f(r_{K} \mathbf{u})$.

\begin{lemma} Given $\mathbf{u}\in V$, let $\mathbf{u}_I$ be defined as above, then
$$
\|\mathbf{u}-\mathbf{u}_I\|_h \lesssim h(|\mathbf{u}|_2+|\nabla\times \mathbf{u}|_2).
$$
\end{lemma}

\begin{proof}
Let $r_h \mathbf{u}$ be the global interpolation operator defined by $r_h \mathbf{u}|_K = r_K \mathbf{u}$.
By triangle inequality,
$$
\|\mathbf{u}-\mathbf{u}_I\|_h\leq \|\mathbf{u}-r_h \mathbf{u}\|_h+\|r_h \mathbf{u} - \mathbf{u}_I\|_h.
$$
By the interpolation error estimate of N\'{e}d\'{e}lec element, we have
$$
\|\mathbf{u} - r_h \mathbf{u}\|_h \lesssim h(|\mathbf{u}|_2+|\nabla\times \mathbf{u}|_2).
$$

Notice that on each tetrahedron $K$,
$$
r_K \mathbf{u} - \mathbf{u}_I = \sum_{f\subset K}\sum M_{mnp} (r_K \mathbf{u} - \mathbf{u}_I)\phi_{mnp},
$$
where $\{\phi_{mnp}\}$ are basis functions on face $f$, and $\{M_{mnp}(\cdot)\}$ are degrees of freedom on face $f$.
Using Lemma \ref{trace_lemma} and $\|q_{np}\|_{L^2(f)} = O(h^2)$, we get
\begin{eqnarray*}
&&2|M_{mnp}(r_K \mathbf{u} - \mathbf{u}_I)|=|M_{mnp}(r_{K'_f} \mathbf{u})-M_{mnp}(r_{K} \mathbf{u})|\\
&&=\frac{1}{|f|^2}\bigg|\int_f(\nabla\times r_{K'_f} \mathbf{u} - \nabla\times r_{K} \mathbf{u} )\times \mathbf{n}\cdot \mathbf{q}_{np}\;dA\bigg|\\
&&\leq \frac{1}{|f|^2}\bigg|\int_f (\nabla\times(r_{K'_f}\mathbf{u} - \mathbf{u})\times \mathbf{n} \cdot \mathbf{q}_{np}\;dA\bigg| +  \frac{1}{|f|^2}\bigg|\int_f (\nabla\times(r_{K}\mathbf{u} - \mathbf{u})\times \mathbf{n} \cdot \mathbf{q}_{np}\;dA\bigg| \\
&&\lesssim \frac{\|\mathbf{q}_{np}\|_{L^2(f)}}{|f|^2}
(\|\nabla\times(r_{K'_f}\mathbf{u} - \mathbf{u})\times \mathbf{n}\|_{L^2(f)} +\|\nabla\times (r_{K}\mathbf{u} - \mathbf{u})\times \mathbf{n}\|_{L^2(f)})\\
&&\lesssim h^{-2}(h^{-1/2}\|\nabla\times(r_{K'_f}\mathbf{u} - \mathbf{u})\|_{0, K\cup K'} + h^{1/2}| \nabla\times (r_{K}\mathbf{u} - \mathbf{u})|_{1,K\cup K'})\\
&&\lesssim h^{-1/2}|\nabla\times \mathbf{u}|_{2,K\cup K'}.
\end{eqnarray*}

Notice $\nabla\lambda_i=O(h^{-1})$, and $\|\phi_{mnp}\|_{0,K}^2=O(h^7)$, by Cauchy-Schwarz inequality, we have
\begin{eqnarray*}
\|r_K \mathbf{u} - \mathbf{u}_I\|_{0,K} &\leq &\left(\sum_{f\subset K}\sum |M_{mnp}(r_K \mathbf{u} - r_I \mathbf{u})|^2\right)^{1/2}\left( \sum_{f\subset K}\sum \|\phi_{mnp}\|^2_{0,K}\right)^{1/2}\\
& \lesssim & h^3|\nabla\times \mathbf{u}|_{2,S(K)},
\end{eqnarray*}
where $S(K)=\cup_{K'\in\mathcal{T}_h, K'\cap K\neq \emptyset} K'$.

Hence,
\begin{equation}
\|r_h \mathbf{u} - \mathbf{u}_I\|_{0,\Omega} = (\sum_K \|r_K \mathbf{u} - \mathbf{u}_I\|^2_{0,K})^{1/2}\lesssim h^3|\nabla\times \mathbf{u}|_{2,\Omega}.
\label{ruuI}
\end{equation}
By inverse inequality, we have
$$
\|\nabla\times (r_h \mathbf{u} - \mathbf{u}_I)\|_{0,\Omega}\lesssim h^2|\nabla\times \mathbf{u}|_{2,\Omega},
$$
$$
\|\nabla(\nabla\times (r_h \mathbf{u} - \mathbf{u}_I))\|_{0,\Omega}\lesssim h|\nabla\times \mathbf{u}|_{2,\Omega}.
$$
Combining these estimates, we get
$$
\|r_h \mathbf{u} - \mathbf{u}_I\|_h\lesssim h|\nabla\times \mathbf{u}|_{2,\Omega},
$$
and the desired estimate follows.
\end{proof}

Remark: We note that the error estimate (\ref{ruuI}) indicates that $r_h \mathbf{u}$ and $\mathbf{u}_I$ are super-close. Such type of estimate can not usually be obtained by the standard scaling argument (using Bramble-Hilbert lemma). In our proof, we made use of the detailed information of the basis functions constructed in the previous section.

\subsubsection{Consistency error estimate}

Given a tetrahedron $K$, in addition to the local interpolation operator $r_K$, we introduce another local interpolation operator $\tilde{r}_K$ corresponding to the first order N\'{e}d\'{e}lec element of the second family, namely, $\tilde{r}_K \mathbf{u}\in (P_1(K))^3\subset R_2(K)$, and
$$
\int_e ((\tilde{r}_K \mathbf{u})\cdot \tau ) \;q\;ds = \int_e
(\mathbf{u}\cdot\tau)\; q\;ds,\;\forall\;q\in P_1(e),\;\forall\;\text{edge}\;e\subset K.
$$

Consider two tetrahedra $K$ and $K'_f$ that share a common face $f$. Given $\mathbf{v}_h\in V_h$, define $\mathbf{v}_K=\mathbf{v}_h|_K$. By definition,
$$
\tilde{r}_{K}\mathbf{v}_{K} = \tilde{r}_{K'_f}\mathbf{v}_{K'_f},\;\text{on face}\;f.
$$
Hence,
%\begin{equation*}
%v_{K}\times n - v_{K'_f}\times n = (v_{K}-\tilde{r}_{K}v_{K})\times n - (v_{K'_f}-\tilde{r}_{K'_f}v_{K'_f})\times n,
%\end{equation*}
\begin{equation}
\sum_K \int_{\partial K} \varphi\cdot [(\tilde{r}_K\mathbf{v}_K)\times \mathbf{n}] \;dA = 0,\;\forall\;\varphi\in H(\text{curl};\Omega),
\label{nedelec face equality}
\end{equation}
where $\mathbf{n}$ is the unit outward normal vector of $\partial K$.

Consider the decomposition (see \cite{Girault:1986fk}):
$$
\mathbf{v}_K = \nabla p_K + \mathbf{w}_K,
$$
where $\text{div}\;\mathbf{w}_K = 0$, $\mathbf{w}_K\cdot \mathbf{n}|_{\partial K}=0$, and $p_K\in P_2(K)$.
The following Lemma \ref{hu-zou-lemma} can be found in \cite{Hu:2004fp}:
\begin{lemma}\
$$
\|\tilde{r}_K \mathbf{w}_K - \mathbf{w}_K\|_{0,K}\lesssim h\|\nabla\times \mathbf{v}_K\|_{0,K}.
$$
\label{hu-zou-lemma}
\end{lemma}

As a consequence of Lemma \ref{hu-zou-lemma}, we have the following estimate.

\begin{lemma}
$$
\|\tilde{r}_K \mathbf{v}_K - \mathbf{v}_K\|_{0,K}\lesssim h\|\nabla\times \mathbf{v}_K\|_{0,K}.
$$
\label{wk lemma}
\end{lemma}

\begin{proof}
Using the interpolation operators defined above, we obtain
$$
\tilde{r}_K\mathbf{v}_K = \tilde{r}_K\nabla p_K + \tilde{r}_K\mathbf{w}_K = \nabla p_K +\tilde{r}_K\mathbf{w}_K.
$$
Hence,
$$
\tilde{r}_K \mathbf{v}_K- \mathbf{v}_K = \tilde{r}_K \mathbf{w}_K - \mathbf{w}_K.
$$

By Lemma \ref{hu-zou-lemma}, we obtain
$$
\|\tilde{r}_K \mathbf{v}_K - \mathbf{v}_K\|_{0,K}\lesssim h\|\nabla\times \mathbf{v}_K\|_{0,K}.
$$
\end{proof}

Now, we can show the following lemma, which is critical for the consistency error estimate.

\begin{lemma} For $\varphi\in H(\text{curl};\Omega)$,
$$
|\sum_{K}\int_{\partial K}\varphi \cdot (\mathbf{v}_h\times \mathbf{n})\;dA
\lesssim h(\|\varphi\|_{0,\Omega}+\|\nabla\times \varphi\|_{0,\Omega})\left(\sum_K\|\nabla\times \mathbf{v}_h\|_{1,K}^2\right)^{1/2}.
$$
\label{critical lemma for consistency}
\end{lemma}

\begin{proof}
By the interpolation error estimates of the N\'{e}d\'{e}lec elements
$$
\|\nabla\times (\tilde{r}_K\mathbf{v}_K -\mathbf{v}_K)\|_{0,K}\lesssim h \|\nabla\times \mathbf{v}_K\|_{1,K},
$$
Lemma \ref{wk lemma}, and Equation (\ref{nedelec face equality}), we have
\begin{eqnarray*}
&&\bigg|\sum_K \int_{\partial K}\varphi\cdot (\mathbf{v}_K\times \mathbf{n})\;dA\bigg|
= \bigg|\sum_K\int_{\partial K}\varphi\cdot [(\tilde{r}_K \mathbf{v}_K-\mathbf{v}_K)\times \mathbf{n}]\;dA\bigg|\\
&&= \bigg|\sum_K\int_K (\nabla\times \varphi)\cdot (\tilde{r}_K\mathbf{v}_K-\mathbf{v}_K)\;dx
- \varphi\cdot [\nabla\times (\tilde{r}_K\mathbf{v}_K-\mathbf{v}_K)]\;dx\bigg| \\
&&\leq \sum_K(\|\nabla\times\varphi\|_{0,K}\|\tilde{r}_K
\mathbf{v}_K-\mathbf{v}_K\|_{0,K}
+\|\varphi\|_{0,K}\|\nabla\times(\tilde{r}_K
\mathbf{v}_K-\mathbf{v}_K)\|_{0,K})\\
&&\lesssim h(\|\varphi\|_{0,\Omega}+
\|\nabla\times\varphi\|_{0,\Omega})
\left(\sum_K\|\nabla\times \mathbf{v}_h\|_{1,K}^2\right)^{1/2}.
\end{eqnarray*}
\end{proof}

Next, we show the consistency error estimate for the nonconforming finite element approximation defined above.

\begin{theorem}
Assume that $\mathbf{u}\in V$ is sufficiently smooth and $\mathbf{v}_h\in V_h$, then
\begin{multline*}
|a_h(\mathbf{u},\mathbf{v}_h)-(\mathbf{f},\mathbf{v}_h)|
 \lesssim h(\| \nabla\times \Delta(\nabla\times \mathbf{u})\|+|\nabla(\nabla\times \mathbf{u})|_1\\+\|(\nabla\times)^2 \mathbf{u}\|+\|\nabla\times \mathbf{u}\|)\left(\sum_K\|\nabla\times
 \mathbf{v}_h\|_{1,K}^2\right)^{1/2}.
\end{multline*}
\end{theorem}

\begin{proof}
By applying integration by parts, we get
\begin{eqnarray*}
&&\;\;\;\;\left(\nabla(\nabla\times \mathbf{u}),\nabla(\nabla\times \mathbf{v}_h)\right)_K \\&=& -(\Delta(\nabla\times \mathbf{u}),\nabla\times \mathbf{v}_h)_K+(\nabla(\nabla\times \mathbf{u})\cdot \mathbf{n},\nabla\times \mathbf{v}_h)_{\partial K}\\
&=&-(\nabla\times\Delta(\nabla\times \mathbf{u}), \mathbf{v}_h)_K
+(\Delta (\nabla\times \mathbf{u}),\mathbf{v}_h\times \mathbf{n})_{\partial K}\\
&&+(\nabla(\nabla\times \mathbf{u})\cdot \mathbf{n},\nabla\times \mathbf{v}_h)_{\partial K},
\end{eqnarray*}
and
$$
(\nabla\times \mathbf{u},\nabla\times \mathbf{v}_h)_K
=((\nabla\times)^2 \mathbf{u},\mathbf{v}_h)_K-(\nabla\times \mathbf{u}, \mathbf{v}_h\times \mathbf{n})_{\partial K}.
$$

Hence,
\begin{eqnarray*}
&&\;\;\;\; a_h(\mathbf{u},\mathbf{v}_h)-(\mathbf{f},\mathbf{v}_h)\\
&&=\sum_{K\in\mathcal{T}_h}[\alpha(\Delta(\nabla\times \mathbf{u}), \mathbf{v}_h \times
\mathbf{n})_{\partial K}+\alpha(\nabla(\nabla\times \mathbf{u})\cdot \mathbf{n},\nabla\times
\mathbf{v}_h)_{\partial K}\\
 && \quad\quad\quad\quad- \beta(\nabla\times \mathbf{u},\mathbf{v}_h\times \mathbf{n})_{\partial K}]\\
&&=\sum_{K\in\mathcal{T}_h}\left[(\alpha \Delta (\nabla\times \mathbf{u})-\beta\nabla\times \mathbf{u},\mathbf{v}_h\times \mathbf{n})_{\partial K}\right]\\
&&\quad\quad+\sum_{K\in\mathcal{T}_h}\left[
\alpha(\nabla(\nabla\times \mathbf{u})\cdot \mathbf{n}, \nabla\times \mathbf{v}_h)_{\partial K}\right].\label{int-by-parts}
\end{eqnarray*}

By Lemma \ref{critical lemma for consistency}, we have
\begin{eqnarray*}
&&\sum_{K\in\mathcal{T}_h}\left[(\alpha \Delta (\nabla\times \mathbf{u})-\beta\nabla\times \mathbf{u},\mathbf{v}_h\times \mathbf{n})_{\partial K}\right]\\
&& \lesssim h(\|\Delta(\nabla\times \mathbf{u})\|_{0,\Omega}+\|\nabla\times \Delta(\nabla\times \mathbf{u})\|_{0,\Omega}+\|\nabla\times \mathbf{u}\|_{0,\Omega}+\|(\nabla\times)^2 \mathbf{u}\|_{0,\Omega}
)\\
&&\hskip 3mm \left(\sum_K\|\nabla\times \mathbf{v}_h\|_{1,K}^2\right)^{1/2}.
\end{eqnarray*}

By Lemma \ref{face_estimate_lemma} and the inter-element continuity of $\nabla\times \mathbf{v}_h$, we get
\begin{eqnarray*}
&&\sum_{K\in\mathcal{T}_h}\left[
\alpha(\nabla(\nabla\times \mathbf{u})\cdot \mathbf{n}, \nabla\times \mathbf{v}_h)_{\partial K}\right]\\
&&\leq \alpha\left| \sum_{K\in\mathcal{T}_h}\sum_{f\subset\partial K}(
\nabla(\nabla\times \mathbf{u})\cdot \mathbf{n}-P_f(\nabla(\nabla\times \mathbf{u})\cdot \mathbf{n}),\nabla\times \mathbf{v}_h-P_f(\nabla\times \mathbf{v}_h))_f\right|\\
&&\lesssim h|\nabla(\nabla\times \mathbf{u})|_{1,\Omega}\left(
|\sum_{K\in\mathcal{T}_h}|\nabla\times \mathbf{v}_h|_{1,K}^2\right)^{1/2}.
\end{eqnarray*}
The theorem follows by combining the above estimates of the two boundary integrals.
\end{proof}

Finally, we have the following convergence result.

\begin{theorem}
Let $\mathbf{u}$ and $\mathbf{u}_h$ be the solutions of the problems (3.10) and (3.16) respectively, then
$$
\norm{\mathbf{u}-\mathbf{u}_h}_{0,h}+\norm{\nabla\times
(\mathbf{u}-\mathbf{u}_h)}_{0,h}+\norm{\nabla(\nabla\times (\mathbf{u}-\mathbf{u}_h))}_{0,h}\lesssim
h \norm{\mathbf{u}}_{4,\Omega}
$$
when $\mathbf{u}\in (H^4(\Omega))^3$.\label{convergence-theorem}
\end{theorem}

\begin{proof}
Using the second Strang lemma,
\begin{eqnarray*}
&&\;\;\;\;\norm{\mathbf{u}-\mathbf{u}_h}_{0,h}+
\norm{\nabla\times
(\mathbf{u}-\mathbf{u}_h)}_{0,h}+
\norm{\nabla(\nabla\times (\mathbf{u}-\mathbf{u}_h))}_{0,h}\\
&&\lesssim \inf_{\mathbf{w}_h\in V_{h}}(\norm{\mathbf{u}-\mathbf{w}_h}_{0,h}+\norm{\nabla\times
(\mathbf{u}-\mathbf{w}_h)}_{0,h}+\norm{\nabla(\nabla\times (\mathbf{u}-\mathbf{w}_h))}_{0,h})
\\
&&\;\;\;\;+ \sup_{\mathbf{w}_h\in V_{h},\mathbf{w}_h\neq
0}\frac{a_h(\mathbf{u},\mathbf{w}_h)-(\mathbf{f},
\mathbf{w}_h)}{\norm{\nabla\times
\mathbf{w}_h}_{1,h}},
\end{eqnarray*}
and previous lemmas, the desired inequality follows.
\end{proof}
%-------------------

\section*{Acknowledgement}
The authors would like to thank Prof. Ludmil Zikatanov and Dr. Luis Chac\'{o}n for many helpful discussions.

\bibliographystyle{amsplain}

\end{document}